\documentclass[12pt]{amsart}
\usepackage{amssymb,latexsym}
\usepackage[utf8]{inputenc}
\usepackage{amscd}
\usepackage{color}
\usepackage{mathtools}
\usepackage[all]{xy}
\usepackage[top=35mm, bottom=25mm, left=30mm, right=30mm]{geometry}
\usepackage{difftrees}
\begin{document}
% One author
\title[Pre-Lie-Rinehart algebra of aromatic trees]
{The universal pre-Lie-Rinehart algebras of aromatic trees}
%\title[shorttitle]{Extending the Grossman-Larson Hopf algebra\\ 
%                   with aromas and forms}
\author{Gunnar Fl{\o}ystad}
\address{Matematisk Institutt\\
        Universitetet i Bergen} 
\email{Gunnar.Floystad@uib.no}

%\urladdr{http://webaddress}
%\thanks{thanks}
% End one author

\author{Dominique Manchon}
\thanks{This work was partially supported by the project Pure Mathematics in Norway, funded by Trond Mohn Foundation and Tromsø Research Foundation}
\address{Laboratoire de Mathematique Blaise Pascal\\
         CNRS and Universite Clermont-Auvergne}
\email{Dominique.Manchon@uca.fr}

%\urladdr{}

\author{Hans Z. Munthe-Kaas}
\address{Matematisk Institutt\\
        Universitetet i Bergen}
\email{Hans.Munthe-Kaas@uib.no}

%\urladdr{http://firstauthorwebaddress}

%\thanks{The firstauthor  ... thanks} 
%\author{secondauthor}
%\address{address2line1\\
%         address2line2\\
%         address2line3}
%\email{secondauthor@email address}
%\urladdr{http://secondauthorwebaddress}
%\thanks{The second author ... thanks}
% End two authors

%\makeatletter
%\@namedef{subjclassname@2020}{%
%  \textup{2020} Mathematics Subject Classification}
%\makeatother

\date{\today}

\begin{abstract}
We organize colored aromatic trees into a pre-Lie-Rinehart algebra (i.e. a flat torsion-free Lie-Rinehart algebra) endowed with a natural trace map, and show the freeness of this object among pre-Lie-Rinehart algebras with trace. This yields the algebraic foundations of aromatic B-series. 
\end{abstract}
\maketitle

%\input /Home/siv7/nmagf/Desktop/macro.tex
% Option 5
% Theorems, corollaries, lemmas, and propositions, in the most 
% emphatic (plain) style; all are numbered separately.
% There is a Main Theorem in the most emphatic (plain) 
% style, unnumbered. There are definitions, in the less emphatic
% (definition) style. There are notations, in the least emphatic
%(remark) style, unnumbered.

\theoremstyle{plain}
\newtheorem{theorem}{Theorem}[section]
\newtheorem{corollary}[theorem]{Corollary}
\newtheorem*{main}{Main Theorem}
\newtheorem{lemma}[theorem]{Lemma}
\newtheorem{proposition}[theorem]{Proposition}
\newtheorem{conjecture}[theorem]{Conjecture}

\theoremstyle{definition}
\newtheorem{definition}[theorem]{Definition}
\newtheorem{fact}[theorem]{Fact}

\theoremstyle{remark}
\newtheorem{notation}[theorem]{Notation}
\newtheorem{remark}[theorem]{Remark}
\newtheorem{example}[theorem]{Example}
\newtheorem{claim}[theorem]{Claim}

%Dominiques egne
\def \restr#1{\mathstrut_{\textstyle |}\raise-6pt\hbox{$\scriptstyle #1$}}
\long\def\ignore#1{}
\def \mop#1{\mathop{\hbox{\rm #1}}\nolimits}
\def \mopl#1{\mathop{\hbox{\rm #1}}\limits}
\def \smop#1{\mathop{\hbox{\sevenrm #1}}\nolimits}
\newcommand{\T}{{\mathcal T}}
\def\arbrea{\,{\scalebox{0.15}{
  \begin{picture}(8,55) (370,-248)
    \SetWidth{2}
    \SetColor{Black}
    \Line(374,-244)(374,-200)
    \Vertex(374,-197){9}
    \Vertex(375,-245){12}
  \end{picture}
}}\,}

\newcommand{\lri}{L}
\newcommand{\emo}{E\ell}
\newcommand{\rri}{R}
\newcommand{\tre}{T}
\newcommand{\vtre}{V}
\newcommand{\aro}{A}
\newcommand{\varo}{B}
\newcommand{\arot}{D}
\newcommand{\eend}{\text{End}}
\newcommand{\di}{d}

%Mine egne.

\newcommand{\psp}[1]{{{\bf P}^{#1}}}
\newcommand{\psr}[1]{{\bf P}(#1)}
\newcommand{\op}{{\mathcal O}}
\newcommand{\opw}{\op_{\psr{W}}}
\newcommand{\go}{\op}

%Initial ideals
\newcommand{\ini}[1]{\text{in}(#1)}
\newcommand{\gin}[1]{\text{gin}(#1)}
\newcommand{\kr}{{\Bbbk}}
\newcommand{\pd}{\partial}
\newcommand{\vardel}{\partial}
\renewcommand{\tt}{{\bf t}}

%Kategorier

\newcommand{\coh}{{{\text{{\rm coh}}}}}

%Modulkategorier

\newcommand{\modv}[1]{{#1}\text{-{mod}}}
\newcommand{\modstab}[1]{{#1}-\underline{\text{mod}}}

\newcommand{\sut}{{}^{\tau}}
\newcommand{\sumit}{{}^{-\tau}}
\newcommand{\til}{\thicksim}

\newcommand{\totp}{\text{Tot}^{\prod}}
\newcommand{\dsum}{\bigoplus}
\newcommand{\dprod}{\prod}
\newcommand{\lsum}{\oplus}
\newcommand{\lprod}{\Pi}

% Algebraer
\newcommand{\La}{{\Lambda}}

\newcommand{\sirstj}{\circledast}

% Knipper
\newcommand{\she}{\EuScript{S}\text{h}}
\newcommand{\cm}{\EuScript{CM}}
\newcommand{\cmd}{\EuScript{CM}^\dagger}
\newcommand{\cmri}{\EuScript{CM}^\circ}
\newcommand{\cler}{\EuScript{CL}}
\newcommand{\clerd}{\EuScript{CL}^\dagger}
\newcommand{\clerri}{\EuScript{CL}^\circ}
\newcommand{\gor}{\EuScript{G}}
\newcommand{\cF}{\mathcal{F}}
\newcommand{\cG}{\mathcal{G}}
\newcommand{\cH}{\mathcal{H}}
\newcommand{\cM}{\mathcal{M}}
\newcommand{\cE}{\mathcal{E}}
\newcommand{\cD}{\mathcal{D}}
\newcommand{\cI}{\mathcal{I}}
\newcommand{\cP}{\mathcal{P}}
\newcommand{\cK}{\mathcal{K}}
\newcommand{\cL}{\mathcal{L}}
\newcommand{\cS}{\mathcal{S}}
\newcommand{\cC}{\mathcal{C}}
\newcommand{\cO}{\mathcal{O}}
\newcommand{\cJ}{\mathcal{J}}
\newcommand{\cU}{\mathcal{U}}
\newcommand{\cR}{\mathcal{R}}
\newcommand{\cQ}{\mathcal{Q}}
\newcommand{\mm}{\mathfrak{m}}

\newcommand{\dlim} {\varinjlim}
\newcommand{\ilim} {\varprojlim}

%Kategorier
\newcommand{\CM}{\text{CM}}
\newcommand{\Mon}{\text{Mon}}

%Kategorieer av komplekser

\newcommand{\Kom}{\text{Kom}}

% Begreper homologisk alebra

\newcommand{\EH}{{\mathbf H}}
\newcommand{\res}{\text{res}}
\newcommand{\Hom}{\text{Hom}}
\newcommand{\inhom}{{\underline{\text{Hom}}}}
\newcommand{\Ext}{\text{Ext}}
\newcommand{\Tor}{\text{Tor}}
\newcommand{\ghom}{\mathcal{H}om}
\newcommand{\gext}{\mathcal{E}xt}
\newcommand{\id}{\text{{id}}}
\newcommand{\im}{\text{im}\,}
\newcommand{\codim} {\text{codim}\,}
\newcommand{\resol}{\text{resol}\,}
\newcommand{\rank}{\text{rank}\,}
\newcommand{\lpd}{\text{lpd}\,}
\newcommand{\coker}{\text{coker}\,}
\newcommand{\supp}{\text{supp}\,}
\newcommand{\Ad}{A_\cdot}
\newcommand{\Bd}{B_\cdot}
\newcommand{\Fd}{F_\cdot}
\newcommand{\Gd}{G_\cdot}

%Avbildninger og andre symbolforkortelser

\newcommand{\sus}{\subseteq}
\newcommand{\sups}{\supseteq}
\newcommand{\pil}{\rightarrow}
\newcommand{\vpil}{\leftarrow}
\newcommand{\rpil}{\leftarrow}
\newcommand{\lpil}{\longrightarrow}
\newcommand{\inpil}{\hookrightarrow}
\newcommand{\pils}{\twoheadrightarrow}
\newcommand{\projpil}{\dashrightarrow}
\newcommand{\dotpil}{\dashrightarrow}
\newcommand{\adj}[2]{\overset{#1}{\underset{#2}{\rightleftarrows}}}
\newcommand{\mto}[1]{\stackrel{#1}\longrightarrow}
\newcommand{\vmto}[1]{\stackrel{#1}\longleftarrow}
\newcommand{\mtoelm}[1]{\stackrel{#1}\mapsto}
\newcommand{\bihom}[2]{\overset{#1}{\underset{#2}{\rightleftarrows}}}
\newcommand{\eqv}{\Leftrightarrow}
\newcommand{\impl}{\Rightarrow}

\newcommand{\iso}{\cong}
\newcommand{\te}{\otimes}
\newcommand{\into}[1]{\hookrightarrow{#1}}
\newcommand{\ekv}{\Leftrightarrow}
\newcommand{\equi}{\simeq}
\newcommand{\isopil}{\overset{\cong}{\lpil}}
\newcommand{\equipil}{\overset{\equi}{\lpil}}
\newcommand{\ispil}{\isopil}
\newcommand{\vvi}{\langle}
\newcommand{\hvi}{\rangle}
\newcommand{\susneq}{\subsetneq}
\newcommand{\sgn}{\text{sign}}

%Notasjonsforkortelser

\newcommand{\xd}{\check{x}}
\newcommand{\ortog}{\bot}
\newcommand{\tL}{\tilde{L}}
\newcommand{\tM}{\tilde{M}}
\newcommand{\tH}{\tilde{H}}
\newcommand{\tvH}{\widetilde{H}}
\newcommand{\tvh}{\widetilde{h}}
\newcommand{\tV}{\tilde{V}}
\newcommand{\tS}{\tilde{S}}
\newcommand{\tT}{\tilde{T}}
\newcommand{\tR}{\tilde{R}}
\newcommand{\tf}{\tilde{f}}
\newcommand{\ts}{\tilde{s}}
\newcommand{\tp}{\tilde{p}}
\newcommand{\tr}{\tilde{r}}
\newcommand{\tfst}{\tilde{f}_*}
\newcommand{\empt}{\emptyset}
\newcommand{\bfa}{{\mathbf a}}
\newcommand{\bfb}{{\mathbf b}}
\newcommand{\bfd}{{\mathbf d}}
\newcommand{\bfl}{{\mathbf \ell}}
\newcommand{\bfx}{{\mathbf x}}
\newcommand{\bfm}{{\mathbf m}}
\newcommand{\bfv}{{\mathbf v}}
\newcommand{\bft}{{\mathbf t}}
\newcommand{\la}{\lambda}
\newcommand{\bfen}{{\mathbf 1}}
\newcommand{\bfe}{{\mathbf 1}}
\newcommand{\ep}{\epsilon}
\newcommand{\en}{r}
\newcommand{\tu}{s}
\newcommand{\Sym}{\text{Sym}}
\newcommand{\nablah}{{\nabla}}
\newcommand{\hbeta}{\hat{\beta}}
\newcommand{\modN}{N}

\newcommand{\ome}{\omega_E}

\newcommand{\bevis}{{\bf Proof. }}
\newcommand{\demofin}{\qed \vskip 3.5mm}
\newcommand{\nyp}[1]{\noindent {\bf (#1)}}
\newcommand{\demo}{{\it Proof. }}
\newcommand{\demodone}{\demofin}
\newcommand{\parg}{{\vskip 2mm \addtocounter{theorem}{1}  
                   \noindent {\bf \thetheorem .} \hskip 1.5mm }}

\newcommand{\lcm}{{\text{lcm}}}

% Simplisielle komplekser

\newcommand{\dl}{\Delta}
\newcommand{\cdel}{{C\Delta}}
\newcommand{\cdelp}{{C\Delta^{\prime}}}
\newcommand{\dlst}{\Delta^*}
\newcommand{\Sdl}{{\mathcal S}_{\dl}}
\newcommand{\lk}{\text{lk}}
\newcommand{\lkd}{\lk_\Delta}
\newcommand{\lkp}[2]{\lk_{#1} {#2}}
\newcommand{\del}{\Delta}
\newcommand{\delr}{\Delta_{-R}}
\newcommand{\dd}{{\dim \del}}
\newcommand{\Del}{\Delta}

%Monomialidealer
\renewcommand{\aa}{{\bf a}}
\newcommand{\bb}{{\bf b}}
\newcommand{\cc}{{\bf c}}
\newcommand{\xx}{{\bf x}}
\newcommand{\yy}{{\bf y}}
\newcommand{\zz}{{\bf z}}
\newcommand{\mv}{{\xx^{\aa_v}}}
\newcommand{\mF}{{\xx^{\aa_F}}}
\newcommand{\kk}{\mathbf k}

%Standard notasjoner
\newcommand{\Symm}{\text{Sym}}
\newcommand{\pnm}{{\bf P}^{n-1}}
\newcommand{\opnm}{{\go_{\pnm}}}
\newcommand{\ompnm}{\omega_{\pnm}}

\newcommand{\pn}{{\bf P}^n}
\newcommand{\hele}{{\mathbb Z}}
\newcommand{\nat}{{\mathbb N}}
\newcommand{\rasj}{{\mathbb Q}}
\newcommand{\bfone}{{\mathbf 1}}

\newcommand{\dt}{\bullet}
\newcommand{\disk}{\scriptscriptstyle{\bullet}}

\newcommand{\cxF}{F_\dt}
\newcommand{\pol}{f}

\newcommand{\Rn}{{\mathbb R}^n}
\newcommand{\An}{{\mathbb A}^n}
\newcommand{\frg}{\mathfrak{g}}
\newcommand{\PW}{{\mathbb P}(W)}
\newcommand{\Ann}{\operatorname{Ann}}

%Hibi
\newcommand{\pos}{{\mathcal Pos}}
\newcommand{\g}{{\gamma}}

%LPres
\newcommand{\Vaa}{V_0}
\newcommand{\Bp}{B^\prime}
\newcommand{\Bpp}{B^{\prime \prime}}
\newcommand{\bbp}{\mathbf{b}^\prime}
\newcommand{\bbpp}{\mathbf{b}^{\prime \prime}}
\newcommand{\bp}{{b}^\prime}
\newcommand{\bpp}{{b}^{\prime \prime}}

%LPII
\newcommand{\oLa}{\overline{\Lambda}}
\newcommand{\ov}[1]{\overline{#1}}
\newcommand{\ovv}[1]{\overline{\overline{#1}}}
\newcommand{\tm}{\tilde{m}}
\newcommand{\po}{\bullet}

\def\CC{{\mathbb C}}
\def\GG{{\mathbb G}}
\def\ZZ{{\mathbb Z}}
\def\NN{{\mathbb N}}
\def\RR{{\mathbb R}}
\def\OO{{\mathbb O}}
\def\QQ{{\mathbb Q}}
\def\VV{{\mathbb V}}
\def\PP{{\mathbb P}}
\def\EE{{\mathbb E}}
\def\FF{{\mathbb F}}
\def\AA{{\mathbb A}}
\def\X{{\mathcal X}}

\newcommand{\C}{{\mathcal C}}
\newcommand{\m}{{\mathfrak m}}
\newcommand{\Der}{\text{Der}}
\renewcommand{\en}{{\mathbf 1}}
\newcommand{\scup}{\sqcup}
\newcommand{\tri}{\rhd}
\newcommand{\End}{\text{End}}
\newcommand{\mato}[1]{\overset{#1}{\mapsto}}
\newcommand{\For}{\text{\it For}}
\tableofcontents

%---------------------------------------------------
%%%%%%%%
\section{Introduction}
In the analysis of structure preserving discretisation of differential equations, series developments indexed by trees are fundamental tools. The  relationship between algebraic and geometric properties of such series have been extensively developed in recent years. The mother of all these series is B-series, introduced the seminal works of John Butcher in the 1960s~\cite{butcher1963coefficients,butcher1972algebraic}. However, the fundamental idea of  denoting analytical forms of differential calculus with trees was conceived already a century earlier by Cayley~\cite{cayley1857}. 

A modern understanding of B-series stems from the algebra of  flat and torsion free connections naturally associated with  locally Euclidean geometries. The  vector fields on $\RR^d$ form a \emph{pre-Lie} algebra $L$ with product given by the connection $\triangleright$ in (\ref{eq:canonical}). The \emph{free pre-Lie algebra} is the vector space spanned by rooted trees with tree grafting as the product~\cite{chapoton2001pre}. A B-series can be defined as an element $B_a$ in the graded completion of the free pre-Lie algebra, yielding infinite series of trees with coefficients $a(t) \in \RR$ for each tree $t$. By the universal property, a mapping  $\ab\mapsto f\in L$, sending the single node tree to a vector field, extends
uniquely to a mapping $B_a\mapsto B_a(f)$, where $B_a(f)$ is an infinite series of vector fields
\[B_a(f) = a(\ab) f + a(\aabb) f\triangleright f + a(\aaabbb)(f\triangleright f )\triangleright f + a(\aababb) \big(f\triangleright (f \triangleright f) - (f\triangleright f )\triangleright f\big)+\cdots .\]

On the geometric side it has recently been shown~\cite{mclachlan2016b} that B-series are intimately connected with (strongly) affine equivariant families of mappings
of vector fields on Euclidean spaces. An infinite family of smooth mappings  $\Phi_n\colon \X\RR^n\rightarrow \X\RR^n$ for $n\in \NN$ has a unique B-series expansion $B_a$ if and only if the family respects all affine linear mappings $\varphi(x)= Ax+b\colon \RR^m\rightarrow \RR^n$. This means that  $f\in \X\RR^n$ being  $\varphi$-related to $g\in \X\RR^m$ implies $\Phi_n(f)$ being  $\varphi$-related to $\Phi_m(g)$. 
Subject to convergence of the formal series we have $\Phi_n(f)=B_a(f)$. 

\emph{Aromatic} B-series is a generalisation which was introduced for the study of volume preserving integration algorithms~\cite{chartier2007preserving,iserles2007b}, more recently studied in \cite{bogfjellmo2019algebraic,munthe2016aromatic}. 
The divergence of a tree is represented as a sum of 'aromas', graphs obtained  by joining  the tree root to any of  the tree's nodes. 
Aromas are connected directed graphs where each node has one outgoing edge. They consist of one cyclic sub-graph with trees attached to the nodes in the cycle. Aromatic B-series are indexed by \emph{aromatic trees},  defined as a tree multiplied by a number of aromas. 

The geometric significance of aromatic B-series is established in~\cite{munthe2016aromatic}. 
Consider a smooth local mapping of vector fields on a finite dimensional vector space, 
 $\Phi\colon \X \RR^d\rightarrow \X \RR^d$. 'Local' means that the support is non-increasing, $\supp(\Phi(f))\subset \supp(f)$. 
 Such a mapping can be expanded in an aromatic B-series if and only if it is equivariant with respect to all affine (invertible) diffeomorphisms $\varphi(x)= Ax+b\colon \RR^d\rightarrow \RR^d$. An equivalent formulation of this result is in terms of the pre-Lie algebra $L = (\X \RR^d,\triangleright)$ defined in the Canonical example of Section~\ref{subsec:preLR}. The isomorphisms of $L$ are exactly the pullback of vector fields by affine diffeomorphisms $\xi(f)= A^{-1}  f\circ \varphi$, hence:
 \begin{theorem}\label{th:1}Let $L$ be the canonical pre-Lie algebra of vector fields on a finite dimensional euclidean space. A smooth local mapping $\Phi\colon L\rightarrow L$ can be expanded in an aromatic B-series if and only if 
 $\Phi\circ \xi = \xi\circ \Phi$  for all pre-Lie isomorphisms $\xi \colon L\rightarrow L$.
 \end{theorem}
 
 This result shows that aromatic $B$-series have a fundamental geometric significance. 
The  question to be addressed in this paper is to understand their {algebraic} foundations.
\emph{In what sense can aromatic B-series be defined as a free object in some category?} Trees represent vector fields and aromas represent scalar functions on a domain. The derivation of a scalar field by a vector field is modelled by grafting the tree on the aromas. 
A suitable geometric model for this is pre-Lie algebroids, defined as Lie algebroids with a flat and torsion free connection~\cite{munthe2020invariant}. Lie algebroids are
vector bundles on a domain together with an 'anchor map', associating sections of the vector bundle with derivations of the ring of  smooth scalar functions. 

The algebraic structure of Lie algebroids is captured through the notion of \emph{Lie--Rinehart algebras}; the aromatic trees form a module over the commutative ring of aromas, acting as derivations of the aromas through the anchor map given by grafting. 
However, it turns out that the operations of divergence of trees and the grafting anchor map  are not sufficient to generate all aromas. Instead, a sufficient set of operations to generate everything is obtained by the graph versions of taking covariant exterior derivatives of vector fields and taking compositions and traces of the corresponding endomorphisms. These operations are well defined on any finite dimensional pre-Lie algebroid. However,  for the Lie--Rinehart algebra of aromas and trees the trace must be defined more carefully, since e.g.\ the identity endomorphism on aromatic trees does not have a well defined trace.

In this paper we define the notion of \emph{tracial pre-Lie Rinehart algebras} and show that the aromatic B-series arise from the free object in this category.

%%%%%%%%
%%%%%%%%
\section{Lie-Rinehart and pre-Lie-Rinehart algebras}
%%%%%%%%
Lie-Rinehart algebras were introduced by George S. Rinehart in 1963 \cite{rinehart1963differential}. They have been thoroughly studied by several authors since then, in particular by J. H\"ubschmann who emphasized their important applications in Poisson geometry \cite{huebschmann1990}. After a brief reminder on these structures, we introduce pre-Lie-Rinehart algebras which are Lie-Rinehart algebras endowed with a flat and torsion-free connection. We also introduce the mild condition of traciality for Lie-Rinehart algebras. The main fact  (Theorem \ref{alg-trace}) states the traciality of any finite-dimensional Lie algebroid over a smooth manifold.
%%%
\subsection{Reminder on Lie-Rinehart algebras}
%%%
Let $\kk$ be a field, and let $R$ be a unital commutative $\kk$-algebra. Recall that a {\it Lie-Rinehart algebra} over $R$ consists of an $R$-module
$L$ and an $R$-linear map
$$\rho:L\mapsto \mop{Der}_{\kk}(R,R)$$
(the \textsl{anchor map}), such that
\begin{itemize}
  \item $L$ is a ${\kk}$-bilinear Lie algebra with bracket $[\![-,-]\!]$,
    \item The anchor map $\rho$ is a homomorphism of Lie algebras,
      \item For $f \in R$ and $X,Y \in L$ the Leibniz rule holds:
\begin{equation}\label{LR-Leibniz}
[\![X,fY]\!]=\bigl(\rho(X)f\bigr)Y+f[\![X,Y]\!].
\end{equation}
\end{itemize}

\begin{remark}
  In the original article \cite{rinehart1963differential},
  G.Rinehart does not state that the anchor
  map should be a Lie algebra homomorphism. However all articles on Lie-Rinehart
  algebras from the last two decades seem to require this. In the much-cited
  article by J.H\"ubschmann
  \cite{huebschmann1990} from 1990 it is not quite clear whether
  this is required,
  but again in later
  articles like \cite{Hu2} from 1998 and onwards, he % J.H\"ubschmann
  explicitly requires 
  the anchor map to be a Lie algebra homomorphism.

  If one does not require the anchor map to be a Lie algebra homomorphism,
  then if $\Ann(L)$ is the annihilator of $L$ in $R$, it is easy to see
  that $L$ will be a Lie-Rinehart algebra over $R/\Ann(L)$,
  with the anchor map being a Lie algebra homomorphism. In
  particular, for Lie algebroids (see Subsection \ref{subsec:liealgebroids}),
  then $\Ann(L) = 0$, and the
  anchor map will automatically be a Lie algebra homomorphism.
  \end{remark}

  A \textsl{homomorphism} $(\alpha, \gamma) : (L,R) \pil (K,S)$ of
  Lie-Rinehart algebras consists of a Lie $\kk$-algebra
  homomorphism $\alpha$ and $\kk$-algebra homomorphism $\gamma$:
  \[ \alpha : L \pil K, \quad \gamma : R \pil S \]
  such that for $f \in R$ and $X \in L$:
  \begin{itemize}
  \item $\alpha(fX) = \gamma(f)\alpha(X)$,
  \item $\gamma((\rho_L(X)\cdot f) = \rho_K(\alpha(X))\cdot \gamma(f).$
  \end{itemize}

  \noindent A \textsl{connection} on a $R$-module $\modN$ is a $R$-linear map
\begin{eqnarray*}
\nabla:L&\longrightarrow&\mop{End}_{\kk}(\modN)\\
X&\longmapsto &\nabla_X
\end{eqnarray*}
such that 
$$\nabla_X(fY)=\bigl(\rho(X).f\bigr)Y+f\nabla_XY.$$
The curvature of the connection is given by
$$R(X,Y):=[\nabla_X,\nabla_Y]-\nabla_{[\![X,Y]\!]}.$$
If $\modN=L$, the torsion of the connection is given by
$$T(X,Y):=\nabla_XY-\nabla_YX-[\![X,Y]\!].$$

The curvature vanishes if and only if $\modN$ is a module over the Lie algebra $L$ (via $\nabla$). In that case, $\modN$ is called a \textsl{module}
over the Lie-Rinehart algebra $(L,R)$. This is equivalent to the map $\nabla$ being a homomorphism of Lie algebras where $\End_{\kk}(\modN)$ is endowed with the commutator as the Lie bracket. In particular the ${\kk}$-algebra $R$ is a module over
the Lie-Rinehart algebra $(L,R)$. 

%\subsection{Connections and $L$-module morphisms}
\medskip
Let $\modN$ be a $R$-module endowed with a connection $\nabla$ with respect to the Lie-Rinehart algebra $(L,R)$. The $R$-module $\mop{Hom}_R(\modN,\modN)$ can be equipped with the connection defined by (where
  $u \in \mop{Hom}_R(\modN, \modN)$, $X \in L$ and $Y \in \modN$)
\begin{equation}\label{endo-connection}
(\nablah_X u)(Y):=\nabla_X\bigl(u(Y)\bigr)-u(\nabla_XY).
\end{equation}
This connection verifies the Leibniz rule
\begin{equation}\label{leibniz}
\nablah_X(u\circ v)=\nablah_Xu\circ v+u\circ\nablah_Xv,
\end{equation}
as can be immediately checked.
\begin{proposition}
If the connection $\nabla$ on $\modN$ is flat, the corresponding connection $\nablah$ on $\mop{Hom}_R(\modN,\modN)$ given by \eqref{endo-connection} is also flat.
\end{proposition}
\begin{proof}
If $\nabla$ is flat on $\modN$, it is well-known that the corresponding $L$-module structure on $\modN$ yields a $L$-module structure on $\mop{Hom}_R(\modN,\modN)$ via \eqref{endo-connection}, hence a flat connection. To be concrete, a direct computation using \eqref{endo-connection} yields
\begin{eqnarray*}
([\nablah_X,\nablah_Y]-\nablah_{[\![X,Y]\!]}u)(Z)&=&\nabla_X\Bigl(\nabla_Y\bigl(u(Z)\bigr)-u(\nabla_YZ)\Bigr)-\nabla_Y(u\nabla_XZ)+u(\nabla_Y\nabla_XZ)\\
&&-\nabla_Y\Bigl(\nabla_X\bigl(u(Z)\bigr)-u(\nabla_XZ)\Bigr)+\nabla_X(u\nabla_YZ)-u(\nabla_X\nabla_YZ)\\
&&-\nabla_{[\![X,Y]\!]}\bigl(u(z)\bigr)+u(\nabla_{[\![X,Y]\!]}Z)\\
&=&([\nabla_X,\nabla_Y]-\nabla_{[\![X,Y]\!]})\bigl(u(Z)\bigr)-u\big(([\nabla_X,\nabla_Y]-\nabla_{[\![X,Y]\!]})(Z)\bigr).
\end{eqnarray*}
\end{proof}
%%%
%%%

\begin{definition}
  Let $(L,R)$ be a Lie-Rinehart algebra.
 % over the unital commutative $\kk$-algebra $R$. We call
An $R$-module $\modN$ is {\it tracial} if there exists a connection $\nabla$ on $\modN$ and a $R$-linear map $\tau:\mop{Hom}_R(\modN,\modN)\to R$ such that
\begin{itemize}
\item $\tau(\alpha\circ\beta)=\tau(\beta\circ\alpha)\hbox{ (trace property)},$
\item $\tau$ is compatible with the connection and the anchor, i.e. \! for any $X\in L$ and $\alpha\in\mop{Hom}_R(\modN,\modN)$ we have
\begin{equation}\label{trace-two}
\tau(\nablah_X\alpha)=\rho(X).\tau(\alpha).
\end{equation}
If $\modN$ is a module over $(L,R)$, this means that $\tau$ is
a homomorphism of $(L,R)$-modules.

\end{itemize}
\end{definition}

\subsection{Aside on manifolds and Lie algebroids}
%%%
\label{subsec:liealgebroids}
Recall that a \textsl{Lie algebroid} on a smooth manifold $M$ is a Lie-Rinehart algebra over the $\mathbb C$-algebra of smooth $\mathbb C$-valued functions on $M$. It is given by the smooth sections of a vector bundle $E$, and the anchor map comes from a vector bundle morphism from $E$ to the tangent bundle $TM$. The terminology "anchor map" and the notation $\rho$ are often used for the bundle morphism in the literature on Lie algebroids.

\begin{theorem}\label{trace-main}
Let $M$ be a finite-dimensional smooth manifold, and let $V$ be a Lie algebroid on $M$. Any finite-dimensional vector bundle $W$ endowed with a $V$-connection is tracial, i.e.\ the $C^\infty(M)$-module $\modN$ of smooth sections of $W$ is tracial with respect to the Lie-Rinehart algebra $L$ of sections of $V$.
\end{theorem}

\begin{proof}
We can consider the fibrewise trace on the algebra $\mop{Hom}_{C^\infty(M)}(\modN,\modN)$ of smooth sections of the endomorphism bundle 
$\mop{End} W$: it is given fibre by fibre by the ordinary trace of an endomorphism of a finite-dimensional vector space. The trace property is obviously verified.\\

To prove the invariance property \eqref{trace-two}, choose two $V$-connections $\nabla^1$ and $\nabla^2$ on $W$. It is well known (and easily verified) that $c_X:=\nabla^2_X-\nabla^1_X$ belongs to $\mop{Hom}_{C^\infty(M)}(\modN,\modN)$, hence is a section of the vector bundle $\mop{End}(W)$. Now for any section $\varphi$ of $\mop{End}(W)$ we have for any $X\in L$ and $\alpha\in\modN$:
\begin{eqnarray*}
\bigl(\nablah^2_X\varphi\bigr)(\alpha)&=&\nabla^2_X\bigl(\varphi(\alpha)\bigr)-\varphi\bigl(\nabla^2_X(\alpha)\bigr)\\
&=&(\nabla^1_X+c_X)\bigl(\varphi(\alpha)\bigr)-\varphi\bigl(\nabla^1_X(\alpha)+c_X(\alpha)\bigr)\\
&=&\bigl(\nablah^1_X\varphi\bigr)(\alpha)+[c_X,\varphi](\alpha).
\end{eqnarray*}
The trace of a commutator vanishes, hence we get
\begin{equation}\label{invariance-tr}
\mop{Tr}(\nablah^2_X\varphi)=\mop{Tr}(\nablah^1_X\varphi).
\end{equation}
In other words, the trace of $\nablah_X\varphi$ does not depend on the choice of the connection. We can locally (i.e. on any open chart of $M$ trivializing the vector bundle $W$) choose the canonical flat connection with respect to a coordinate system, namely
$$\nabla^0_X\alpha:=\bigl(\rho(X)\alpha_1,\ldots,\rho(X)\alpha_p\bigl)$$
for which \eqref{trace-two} is obviously verified (here $p$ is the dimension of the fibre bundle $W$).  Hence, from \eqref{invariance-tr}, we get that \eqref{trace-two} is verified for any choice of connection $\nabla$. 
\end{proof}

%%%
\subsection{Tracial Lie-Rinehart algebras}
%%%
When the module $\modN$ is the Lie-Rinehart algebra itself, it may be convenient to restrict the algebra on which the trace is defined:
\begin{definition}\label{diff}
Suppose that $\modN=L$, and let us introduce the $\kk$-linear operator
$$d:L\longrightarrow \mop{Hom}_R(L,L)$$
defined by
\begin{equation}
dX(Z):=\nabla_ZX.
\end{equation}
Let the \textsl{algebra of elementary $R$-module endomorphisms} be the
$R$-module subalgebra of $\mop{Hom}_R(L,L)$ generated by $\{\nabla_{Y_1}\cdots\nabla_{Y_n}dX \colon X,Y_1,\ldots,Y_{n}\in L\}$. It will be denoted by $\emo_R(L,L)$.
\end{definition}
\begin{remark}
  The Leibniz rule \eqref{leibniz} implies that the $(L,R)$-module structure on
  $\Hom_R(L,L)$ (via the connection $\nabla$) restricts to  $\emo_R(L,L)$,
  making this an $(L,R)$-submodule of $\Hom_R(L,L)$. 
  %As suggested by the notation, $d$ can be extended as a differential of some complex. We will return to this point in a forthcoming article.
\end{remark}
\begin{definition}
A Lie-Rinehart algebra $L$ over the unital commutative $\kk$-algebra $R$ is \textsl{tracial} if there exists a connection $\nabla$ on $L$ and a $R$-linear map $\tau:\emo_R(L)\to R$ such that
\begin{itemize}
\item $\tau(\alpha\circ\beta)=\tau(\beta\circ\alpha)\hbox{ (trace property)},$
\item $\tau$ is a homomorphism of $L$-modules, i.e. for any $X\in L$ and $\alpha\in \emo_R(L)$ we have
\begin{equation}\label{trace-twotwo}
\tau(\nabla_X\alpha)=\rho(X).\tau(\alpha).
\end{equation}
\end{itemize}

In this case the \textsl{divergence} on $L$ is the composition
$\operatorname{Div} = \tau \circ d$ of
\[ L \mto{d} \emo_R(L) \mto{\tau} R. \]
\end{definition}
\begin{corollary}\label{alg-trace}
  Any finite-dimensional Lie algebroid is tracial for its natural canonical
  trace map.
\end{corollary}
\begin{proof}
It is an immediate consequence of Theorem \ref{trace-main}.
\end{proof}

We also have a an analog for the differential of a function, the
first term in the De Rham complex:
\[ d : R \pil \Hom_R(L,R), \quad f \mapsto ( X \mapsto \rho(X)(f) ) .\]
Given an element $Y$ in $L$, we get a map in $\Hom_R(L,L)$ denoted
$df \cdot X$:
\begin{equation} \label{eq:LR-DRL}
  X \mapsto \rho(X)(f)\cdot Y.
  \end{equation}

  \subsection{pre-Lie-Rinehart algebras}\label{subsec:preLR}
%%%
\begin{definition}
  A \textsl{pre-Lie-Rinehart algebra} is a Lie-Rinehart algebra $L$ endowed with a flat torsion-free connection
  \[ \nabla : L \pil \End_{\kk}(L,L). \]
\end{definition}
We have then, with the notation $X\rhd Y:=\nabla_XY$:
\begin{itemize}
\item $[\![X,Y]\!]=X\rhd Y-Y\rhd X$,
\item $X\rhd(Y\rhd Z)-(X\rhd Y)\rhd Z=Y\rhd(X\rhd Z)-(Y\rhd X)\rhd Z$ (left pre-Lie relation).
\end{itemize}

A \textsl{module} over the pre-Lie-Rinehart algebra is the same as a module over
the underlying Lie-Rinehart algebra. If $N$ is a module and $n$ and element,
we
write $X \rhd n$ for $\nabla_X n$. In particular for $f \in R$ the action of the
anchor map $\rho(X).f$ is written $X \rhd f$.

\medskip
\textbf{Canonical example}  \cite{cayley1857}: Let $\kk=\mathbb R$, let $R=C^\infty(\mathbb R^d)$ and let $L$ be the space of smooth vector fields on $\mathbb R^d$. Let $X,Y\in L$, which are written in coordinates:
$$X=\sum_{i=1}^d f^i\partial_i,\hskip 12mm Y=\sum_{j=1}^d g^j\partial_j.$$
Then
\begin{equation}\label{eq:canonical}X\rhd Y=\sum_{j=1}^d\left(\sum_{i=1}^d f^i(\partial_ig^j)\right)\partial_j.\end{equation}
%%%
%\begin{remark} Continuing the canonical example of Subsection
%\ref{subsec:preLR}
For a vector field
    $X = \sum_{i=1}^d f^i \partial_i$, the endomorphism $dX$ sends
   \[ \sum_{i=1}^d g^j\partial_j \mapsto \sum_{i,j = 1}^d g^j \partial_j(f^i)
      \partial_i. \]
    In particular $\partial_j \mapsto \sum_{i = 1}^d \partial_j(f^i)
    \partial_i$, so the trace of $dX$ is
    the divergence $\sum_{i = 1}^d \partial_i(f^i)$.
  %  Generally for any finite dimensional Lie algebroid with a connection we have a
 %   notion of divergence $\operatorname{Div} = \tau \circ d$.

%%%
%\subsection{Tracial pre-Lie-Rinehart algebras}
%%%
\begin{proposition}\label{E-PL}
  In a pre-Lie-Rinehart algebra $L$, the algebra $\emo_R(L)$ of elementary
  module homomorphisms is generated by $\{dX \colon X\in L\}$.
\end{proposition}
\begin{proof}
In view of Definition \ref{diff}, we first show that for any $X,Y\in L$, the endomorphism $\nabla_Y(dX)$ is obtained by linear combinations of products of endomorphisms of the form $dZ,\, Z\in L$. It derives immediately from the left pre-Lie relation, via the following computation (recall \eqref{endo-connection}):
\begin{eqnarray*}
(\nabla_YdX)(Z)&=&\nabla_Y\bigl(dX(Z)\bigr)-dX(\nabla_YZ)\\
&=&\nabla_Y(Z\rhd X)-(\nabla_YZ)\rhd X\\
&=&Y\rhd(Z\rhd X)-(Y\rhd Z)\rhd X\\
&=&Z\rhd(Y\rhd X)-(Z\rhd Y)\rhd X\\
               &=&\big(d(Y\rhd X)-dX\circ dY\bigr)(Z).                   
\end{eqnarray*}
To then show further that $\nabla_{Y_1} \nabla_{Y_2} dX $ is a linear combination
of products of endomorphisms, we use the Leibniz rule
\[ \nabla_{Y_1} (dX \circ dY_2) = (\nabla_{Y_1} dX) \circ dY_2 +
  dX \circ \nabla_{Y_1} dY_2.\]
In this way we may continue.
\end{proof}

\begin{proposition} \label{prop:LR-dxyr} Let $L$ be a pre-Lie-Rinehart algebra.
  \begin{itemize}
    \item[a.]
  For any $X,Y \in L$:
  \begin{equation*} \label{eq:LR-dxy}
 X \rhd dY = d(X\rhd Y) -  dY\circ dX.
\end{equation*}
\item[b.] For  $f \in R$ (recall \eqref{eq:LR-DRL} for $df \cdot Y$):
\begin{equation*} \label{eq:LR-drx}
  X \rhd (df \cdot Y) = df \cdot (X \rhd Y) + d(X \rhd f) \cdot Y -
  (df \cdot Y) \circ dX.
\end{equation*}
\end{itemize}
% In paricular, if $L$ is tracial one has:
%\begin{equation}
%\operatorname{Div}(x\rhd y)=x \rhd (\operatorname{Div} y)+\tau(dy\circ dx).
%\end{equation}
%%where $\operatorname{Div}$ is a shorthand notation for $\tau\circ d$.
\end{proposition}

\begin{proof} a. Using \eqref{endo-connection}:
  \begin{eqnarray*}
    (X \rhd dY) (Z) & = & X \rhd (Z \rhd Y) - (X \rhd Z) \rhd Y \\
    & = & Z\rhd (X \rhd Y) - (Z \rhd X) \rhd Y \\
    & = & d (X \rhd Y) (Z) - dY \circ dX (Z)
    \end{eqnarray*}
\medskip
b. Again using \eqref{endo-connection} this map sends $Z$ to:
\begin{eqnarray*}  \big( X \rhd (df \cdot Y) \big) (Z)& = &
X \rhd \big((Z \rhd f)Y\big) - (df \cdot Y) (X \rhd Z) \\
\text{(connection property)}
&= & \big(X \rhd (Z \rhd f)\big)Y + (Z \rhd f)(X \rhd Y) -
\big((X \rhd Z) \rhd f\big)Y \\
  &= & (Z \rhd (X \rhd f))Y - ((Z \rhd X) \rhd f)Y + (Z \rhd f)(X \rhd Y).
\end{eqnarray*}
This is the map: 
\[ d (X \rhd f)\cdot Y - (df \cdot Y) \circ dX + df \cdot (X \rhd Y). \]
\end{proof}

\begin{definition} \label{def:LR-TF}
  Let $(L,R)$ and $(K,S)$ be tracial pre-Lie-Rinehart algebras, and
  $(\alpha, \gamma) : (L,R) \pil (K,S)$ a homomorphism of pre-Lie-Rinehart
  algebras. For each $X \in L$ there is a commutative diagram:
  \[ \xymatrix{L \ar[r]^{dX} \ar[d]_{\alpha}  & L \ar[d]^{\alpha} \\
    K \ar[r]^{d\alpha(X)} & K.} \]
  An elementary endomorphism $\phi : L \pil L$ is an $R$-linear combination
  of compositions $dX_1 \circ \cdots \circ dX_r$. It induces an elementary
  endomorphism $\psi : K \pil K$ which is the corresponding $R$-linear
  combination of compositions of $d\alpha(X_i)$'s.
  {\em Note:} This $\psi$ may not be unique since
  expressing $\phi$ as an $R$-linear combination of compositions may not be
  done uniquely. For instance it could be that $dX$ is the zero map, while
  $d\alpha(X)$ is {\it not} the zero map.

  The homomorphism $(\alpha,\gamma)$ is a {\it homomorphism of tracial
    pre-Lie algebras} if for each elementary endomorphism $\phi$
  the trace of $\phi$ maps to the trace of $\psi$:
  $\gamma(\tau \phi) = \tau \psi$. (This is regardless of which
  $\psi$ that corresponds to $\phi$.)
\end{definition} 
  
%%%%%%%%
\section{Aromatic trees}
%%%%%%%%
 In this section we define rooted trees, aromas and aromatic trees, the latter being the relevant combinatorial objects for building up the free pre-Lie-Rinehart algebra.
%%%
\subsection{Rooted trees and aromas}\label{sect:tf}
%%%
Let $\C $ be a finite set, whose elements we shall think of as colors.
We introduce some notation:

\begin{definition} $\tre_\C$ is the vector space freely generated by rooted trees whose vertices are colored 
  with elements of $\C$. We denote by $\vtre_\C$ the vector space freely generated by pairs $(v,t)$ where $t$ is a $\C$-colored tree and
  $v$ is a vertex of $t$.
\end{definition}

There is an injective map
\[ V_\C \pil \End_{\kk}(\tre_\C), \quad (v,t) \mapsto
  (s \mapsto s \rhd_v t) \]
where $\rhd_v$ is grafting the root of $s$ on the
vertex $v$.

The composition $\beta \circ \alpha$ of maps $\beta, \alpha$ in
$\End_{\kk}(\tre_\C)$ induces a multiplication (composition) $\circ$
on $V_C$ given by $(v,t) \circ (u,s) = (u,s \rhd_v t)$.
We may then identify $\vtre_{\C}$ as a $\kk$-subalgebra of
$\End_{\kk}(T_\C)$. Then
\begin{equation} \label{eq:TTT}
  d : T_\C \pil \End_{\kk}(T_\C), \quad t \mapsto \sum_{v \in t} (v,t).
  \end{equation}

\begin{definition}
A connected 
directed graph with vertices colored by $\C$ and where each vertex has precisely
one outgoing edge is called a {\it $\C$-colored aroma} or just an {\it aroma}
since we will only consider this situation.
It consists of a central cycle with trees attached to the vertices of this
cycle. The arrows of each tree are oriented towards the cycle, which will be oriented counterclockwise by convention when an aroma is drawn in the two-dimensional plane. We let $\aro_\C$ be the vector space freely generated by $\C$-colored
aromas. See Figure \ref{fig:aroma}
where the first four connected graphs are aromas.

\end{definition}

Now consider the linear map
\begin{equation} \label{eq:trmap}
\tau : \vtre_{\C} \mto{} \aro_{\C} \end{equation}
which maps the pair $(v,t)$ to an 
aroma by joining the root of $t$ to the vertex $v$.
% The corresponding linear extension from $\T^1_{\C}$ to $\kkA_{\C}$, which will also be denoted by $\tau$, is the universal trace of the algebra $\T^1_{\C}$, namely:

\begin{lemma}\label{univ-trace}
  The vector space $\aro_{\C}$ spanned by the aromas can be naturally identified with the quotient $\vtre_{\C}/[\vtre_{\C},\vtre_{\C}]$, where $[\vtre_{\C},\vtre_{\C}]$ is the vector space spanned by the commutators in $\vtre_{\C}$, so that the map $\tau$ becomes the natural projection from $\vtre_{\C}$ onto
  $\vtre_{\C}/[\vtre_{\C},\vtre_{\C}]$.
\end{lemma}

\begin{proof} An aroma has a unique interior cycle. Let $r_1, \cdots r_n$ be
  the vertices on this cycle. At $r_{j-1}$ there is a tree $t^\prime_j$ with root $r_{j-1}$.
  Let $t_j$ be this tree with $r_{j-1}$ grafted onto $r_j$, so $r_j$ is
  the root of $t_j$. 
  % We can draw the aroma  on the two-dimensional plane as a counterclockwise oriented cycle of length $n\ge 1$ with one tree $t_j$ attached to vertex number $j$ of the cycle, namely:
We may the write the aroma as:
  \begin{equation}
a=\tau\big((r_1,t_1)\circ \cdots \circ (r_n,t_n)\big).
\end{equation}
%where $r_j$ is the root of $t_j$.
and this is invariant under any cyclic permutation of the elements $(r_i,t_i)$.
On the other hand, any tree $t$ with marked point $v$ admits the decomposition:
\begin{equation}
(v,t)=(v_1,t_1) \circ \cdots \circ (v_j,t_j)
\end{equation}
where $v_1$ (resp. $v_j$) is the root of $t$ (resp. the marked vertex $v$) and $(v_1,v_2,\ldots,v_j)$ is the path from the root to $v$ in $t$. Each vertex $v_i$ of this path is the root of the tree $t_i$. Now if
$(v',t')=(v_{j+1},t_{j+1})\circ \cdots \circ (v_{j+k},t_{j+k})$ is another tree with
marked vertex, we have
\begin{eqnarray*}
(v,t)\circ (v',t')&=&(v_1,t_1)\circ \cdots \circ (v_{j+k},t_{j+k}) \hbox{ and}\\
  (v',t')\circ (v,t)&=&(v_{j+1},t_{j+1})\circ \cdots \circ
                        (v_{j+k},t_{j+k})\circ (v_1,t_1)\circ \cdots \circ (v_j,t_j).
\end{eqnarray*}
The trace property
$\tau\big((v,t) \circ (v',t')\big)=\tau\big((v',t')\circ (v,t)\big)$
is then obvious by cyclic invariance of the decomposition of an aroma. Now any aroma is the image by $\tau$ of at most $n$ trees with marked points, where $n$ is the length of the cycle. It is clear that two such trees admit the same decomposition as above modulo cyclic permutation, which implies that they differ by a commutator. Now we have to prove that any element $T\in \vtre_{\C}$ with $\tau(T)=0$ is a linear combination of commutators. Decomposing $T$ in the basis of trees with one marked points:
\begin{eqnarray*}
T&=&\sum_{(v,t)}\alpha_{(v,t)}(v,t)\\
&=&\sum_{a \scalebox{0.7}{\hbox{ aroma}}}\ \sum_{(v,t),\, \tau(v,t)=a}\alpha_{(v,t)}(v,t),\\
\end{eqnarray*}
from $\tau(T)=0$ we get $\sum_{(v,t),\, \tau(v,t)=a}\alpha_{(v,t)}=0$ for any aroma $a$. Hence the sum $\sum_{(v,t),\ \tau(v,t)=a}\alpha_{(v,t)}(v,t)$ is a sum of commutators for any aroma $a$, which proves that $T$ is also a sum of commutators.
\end{proof}

 The canonical embedding of $\C$ into $\tre_{\C}$ is given by $c\mapsto \bullet_c$. It is well-known \cite{chapoton2001pre, dzhumadildaev2002trees} that $\tre_{\C}$ with grafting of trees as the operation $\rhd$ is
 the free pre-Lie algebra on the set $\C$. Then $(\tre_\C,k)$ becomes
 a pre-Lie-Rinehart algebra, with anchor map zero.

\begin{lemma} \label{lem:EC}
  The algebra $\emo_{k}(\tre_{\C})$ of elementary module morphisms of Definition \ref{diff}, for the pre-Lie-Rinehart algebra $(\tre_\C,k)$,
  coincides with the algebra $\vtre_{\C}$ of  trees with one marked point.
\end{lemma}
\begin{proof}
  Since $dt = \sum_{v \in t}(v,t)$ by \eqref{eq:TTT}
  %by Proposition \ref{E-PL}
  we need only to show that each
marked tree $(v,t)$ is in $\emo_{k}(\tre_{\C})$. Let $t_v$ be the
  subtree of $t$ which has $v$ as root. If $v$ is not the root of $t$, it is
  attached to a node $w$. Take $t_v$ away from $t$ and let $t^\prime$ be the
  resulting tree.
  Then $(v,t) = (w,t^\prime) \circ (v,t_v)$. We now show by
  induction on the number of nodes of i) $|t|$ and ii) $|t_v|$, that $(v,t)$
  is in $\emo_{k}(\tre_{\C})$.

  i) If the marked tree is $(\bullet_c, \bullet_c)$, then
  it is $d(\bullet_c)$ and is in $\emo_{k}(\tre_{\C})$.

  ii) If $v$ is a root, then
  \[ d(t) = (v,t) + \sum_{w \neq v}(w,t).\]
  The left term is in $\emo_{k}(\tre_{\C})$, and the right term
  also by induction.
  
  iii) If $v$ is not a root then both $(w,t^\prime)$ and $(v,t_v)$ are in
  $\emo_{k}(\tre_{\C})$ by induction, and so also $(v,t)$.
\end{proof}

As a result $V_\C$ is a $T_\C$-module by grafting:
\[ s \rhd (v,t) = (v,s \rhd t), \]
where the latter is a sum of pairs $(v,t_i)$ coming from that $s \rhd t$ is
a sum of trees $t_i$. The aromas $A_\C$ also form a $T_\C$-module
by grafting the trees on all vertices in an aroma. Lastly the map
$\tau : V_\C \pil A_\C$ is a $T_\C$-module map.

\subsection{The free pre-Lie algebra}
Let $L$ be a pre-Lie algebra over the field $\kk$.
The pre-Lie algebra $T_\C$ has the universal property
that given any map $\C \pil L$ there is a unique morphism of pre-Lie
algebras $\tre_{\C} \pil L$ such that the diagram below commutes:
\[ \xymatrix{
  & \C \ar[dr] \ar[dl] & \\
\tre_{\C} \ar[rr]^\alpha & & L.}
\]

\begin{theorem}  \label{thm:TEA}
  Let $L$ be a tracial pre-Lie-Rinehart algebra over the $\kk$-algebra $R$, with trace map $\eend_R(L,L) \mto{\tau} R$.
  \begin{itemize}
  \item[a.] Given a set map $\psi:\C  \pil L$, the unique pre-Lie algebra
    homomorphism
$\tre_{\C} \mto{\alpha} L$ such that $\alpha(\bullet_c)=\psi(c)$ for any $c\in\C$ induces a unique 
morphism of associative algebras $\vtre_\C \mto{\beta} \emo_R(L)$ such that
the following diagram commutes:
\begin{equation} \label{eq:univ-PE} \begin{CD}
\tre_{\C} @>{d}>> \vtre_\C \\
@V{\alpha}VV @V{\beta}VV \\
L   @>{d}>> \emo_{R}(L).
   \end{CD}
\end{equation}

\item[b.] Moreover, there is a unique linear map $\gamma$
extending the diagram \eqref{eq:univ-PE}
to a commutative diagram
\begin{equation*} \label{eq:univ-TEA} \begin{CD}
\tre_{\C} @>{d}>> \vtre_\C @>{\tau}>> \aro_{\C}\\
@V{\alpha}VV  @V{\beta}VV @V{\gamma}VV \\
L   @>{d}>> \emo_R(L) @>{\tau}>> R.
   \end{CD}
\end{equation*}
\item[c.] These maps fulfill the following for an aroma $a$, tree $t$, and
$\phi \in V_\C \sus \End_{\kk}(T_\C)$:
\begin{itemize}
\item[i.] $\beta(\phi)(\alpha(t)) = \alpha(\phi(t))$,
\item[ii.] $\beta(t \rhd \phi) = \alpha(t) \rhd \beta(\phi)$,
\item[iii.] $\gamma(t \rhd a) = \alpha(t) \rhd \gamma(a)$.
  \end{itemize}
\end{itemize}
\end{theorem}

\begin{proof}
{\bf Part a.} Any tree $(v,t)$ with one marked point different from the root can be written as
\begin{equation}
(v,t)=(v'',t'') \circ (v,t'),
\end{equation}
where the associative product on $\vtre_{\C}$ has been described in Paragraph \ref{sect:tf}. Here $t'$ is any upper sub tree containing the marked vertex $v$, and $t''$ is the remaining tree, on which the marked vertex $v''$ comes from the vertex immediately below the root of $t'$. We then proceed by induction on the number of vertices: if $t$ is reduced to the vertex $v$ colored by $c\in \mathcal C$, we obviously have
\begin{equation}
\beta(v,t)(x)=d\alpha(\bullet_c)(x)=x\rhd\alpha(\bullet_c)
\end{equation}
for any $x\in L$. Suppose that the map $\beta$ has been defined for any tree up to $n$ vertices. Now if $t$ has $n+1$ vertices and one marked vertex $v$ different from the root, we must have:
\begin{equation}
  \beta(v,t)=\beta\big((v'',t'') \circ (v,t')\big)=\beta(v'',t'') \circ
  \beta(v,t').
\end{equation}
It is easily seen that this does not depend on the choice of the decomposition. Indeed, if $v$ is not the root of $t'$, then  $(v,t')=(v',s')(v,s)$ where $s$ is the subtree with root $v$, and $s'$ is the remaining tree inside $t'$. The vertex $v'$ comes from the vertex immediately below $v$ in $t'$. We have then
$$\beta(v,t)=\beta(v'',t'') \circ \beta(v',s') \circ \beta(v,s)=\beta(v',\widetilde t) \circ \beta(v,s)$$
where $\widetilde t$ is obtained by grafting $s'$ on $t''$ at vertex $v''$. Hence any decomposition boils down to the unique one with minimal upper tree, for which the marked vertex is the root. Now if $t$ has $n+1$ vertices and if the marked vertex is the root, we must define $\beta(v,t)$ as follows:
\begin{equation}
\beta(v,t)=d\alpha(t)-\sum_{v'\neq v}\beta(v',t).
\end{equation}

\medskip
\noindent {\bf Part b.}
The map $\beta$ is an algebra morphism, hence induces a map 
$$\overline\beta: \vtre_{\C}/[\vtre_{\C},\vtre_{\C}]\to \hbox{Hom}_R(L,L)/[\hbox{Hom}_R(L,L),\,\hbox{Hom}_R(L,L)].$$
The map $\tau$ of the bottom line of the diagram being a trace, it induces a map $\overline\tau:\hbox{Hom}_R(L,L)/[\hbox{Hom}_R(L,L),\,\hbox{Hom}_R(L,L)]\to L$. In view of Lemma \ref{univ-trace}, the map $\gamma:=\overline \tau\circ\overline\beta$ then makes Diagram \eqref{eq:univ-PE} commute.

\medskip
\noindent {\bf Part c.}
We prove first ii. Let $t,s$ be trees and first consider $\phi = ds$. Recall
by Proposition \ref{prop:LR-dxyr}a:
\begin{eqnarray*}
 t \rhd ds &= & d(t \rhd s) - ds \circ dt \\
  \beta(t \rhd ds)
  &= & \beta d(t\rhd s) - \beta d(s) \circ \beta d(t) \notag \\
  &= & d \alpha(t \rhd s) - d \alpha(s) \circ d \alpha(t) \notag \\
  &= & d (\alpha(t) \rhd  \alpha(s)) - d\alpha(s) \circ d \alpha(t) \notag
%\label{eq:aroma-dats}
\end{eqnarray*}
Again by Proposition \ref{prop:LR-dxyr}a this equals:
%\[ \alpha(t) \rhd d \alpha(s) = d(\alpha(t) \rhd \alpha(s)) - d\alpha(s) \circ
%  d \alpha(t), \]
\[ \alpha(t) \rhd d \alpha(s) = \alpha(t) \rhd \beta(ds). \]

Now if part ii holds for $\phi_1$ and $\phi_2$ it is
an immediate computation to verify that
\[ \beta\big(t \rhd (\phi_1 \circ \phi_2)\big) = \alpha(t) \rhd
  \beta(\phi_1 \circ \phi_2), \]
thus showing part ii.

%%%%%%%%%%% Begynn er videre
\medskip
Part i is shown in a similar way. First consider $\phi = du$.
Then
\begin{eqnarray*}
  \beta(du)(\alpha(t))  = & d \alpha(u) (\alpha(t)) \\
  = & \alpha(t) \rhd \alpha(u) \\
  = & \alpha (t \rhd u) \\
  = & \alpha( du(t)).
\end{eqnarray*}
Lastly one may show that if i holds for $\phi_1$ and $\phi_2$,
it holds for their composition.
\medskip
For Part iii the aroma $a$ is an image $\tau \phi$ for a marked
tree $\phi$. Hence:
\begin{eqnarray*}
  \gamma(t \rhd a) = & \gamma(t \rhd \tau \phi) \\
\text{(trace is an $L$-homomorphism)}  = & \gamma \tau (t \rhd \phi) \\
  = & \tau \beta(t \rhd \phi) \\
  \text{(use Part ii.)}  = & \tau(\alpha(t) \rhd \beta(\phi)) \\
  \text{(trace is an $L$-homomorphism)} = & \alpha(t) \rhd \tau \beta(\phi) \\
  = & \alpha(t) \rhd \gamma\tau(\phi) \\
  = & \alpha(t) \rhd \gamma(a).
      \end{eqnarray*}
\end{proof}

\subsection{The pre-Lie-Rinehart algebra of aromatic trees}

\begin{definition}
Let $R_{\C}$ be the vector space freely generated by $\C$-colored directed graphs (not necessarily connected) where
each vertex has precisely one outgoing edge.
Such a directed graph is a multiset of aromas, and we call it a multi-aroma.
%Let $R_{\C}$ be the vector space $\kkB_{\C}$ freely generated by $B_{\C}$.
\end{definition}

The vector space $R_{\C}$ has a commutative unital $\kk$-algebra
structure coming from the monoid structure on multisets of aromas.  Note that $R_{\C}$ is the symmetric algebra $\Sym_{\kk}(A_{\C})$ on the vector space of
$\C$-colored aromas.

\begin{remark} Denote $[n] = \{1,2,\ldots, n\}$. 
In the case of one color, a multi-aroma on $n$ vertices is simply a map $f : [n] \pil [n]$. 
More precisely the multi-aromas identify as orbits of such maps
by the action of the symmetric group $S_n$. 
\end{remark}

\begin{definition}
  Denote $R_\C \te_{\kk} T_\C$ by $L_\C$.
  As a vector space it has as basis all expressions $r \te_{\kk} t$
  where $r$ is a multi-aroma
  and $t$ is a tree. For short we write this as $rt$ and call it
  an \textsl{aromatic tree}, \cite{munthe2016aromatic}.
  See Figure \ref{fig:aroma}.

\end{definition}

\begin{figure}\vspace{-0.2cm}
\centerline{\includegraphics[width=6cm]{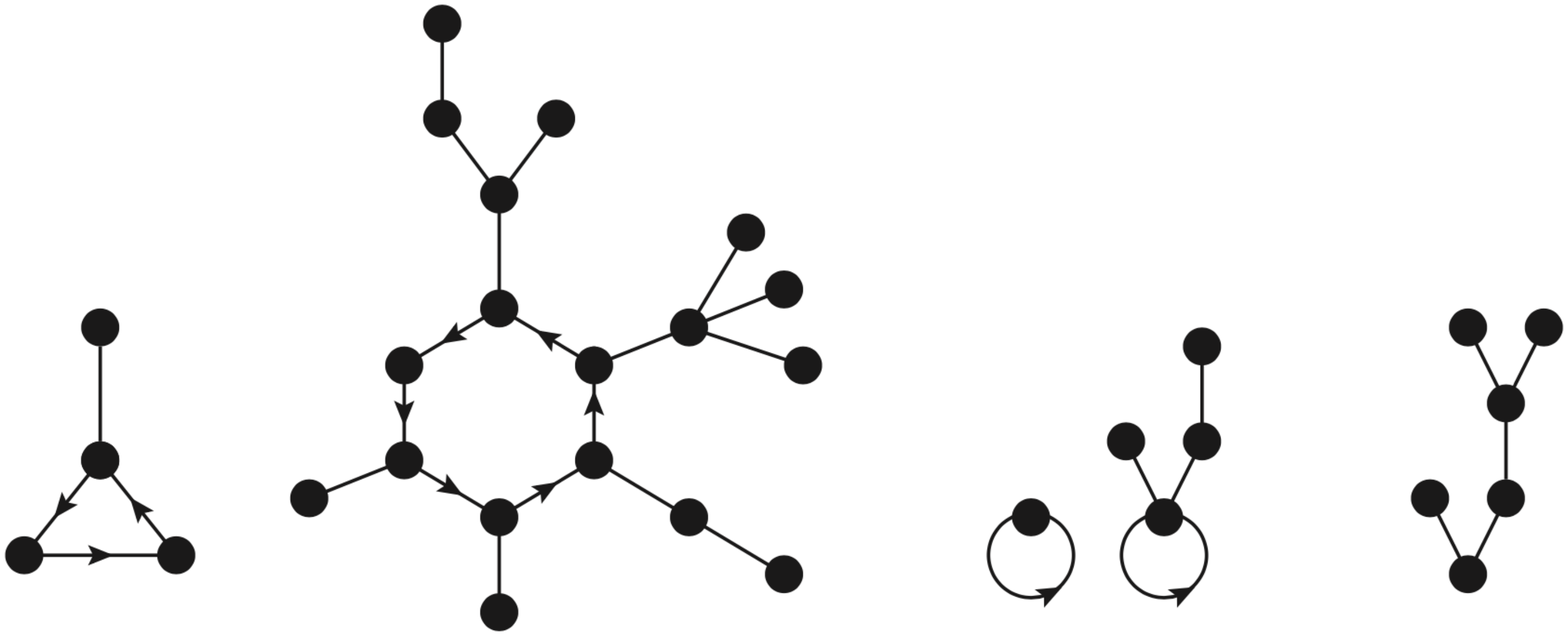}}
\vspace{-1cm}
\caption{\small An example of aromatic tree made of four aromas and a rooted tree}
\label{fig:aroma}
\end{figure}

On $\lri_C$ we have the product
\begin{eqnarray*} 
\lri_{\C}\times \lri_{\C}&\longrightarrow  &\lri_{\C}\\
(qs,rt)&\longmapsto  & \nabla_{qs} rt = qs\rhd rt
\end{eqnarray*}
given by grafting the root of the tree $s$ on any vertex
of the aromatic tree $rt$ and summing up. Similarly we can graft an aromatic 
tree on an aroma. From this we get induced maps:
\begin{align*}
  \nablah : \lri_\C & \pil \Hom_{\kk}(\lri_\C, \lri_\C) &  \quad \rho : \lri_\C & \pil \Der_{\kk}(R_\C,R_\C) &
d : \lri_\C & \pil \Hom_{R_C}(\lri_\C,\lri_\C) \\
rt & \mapsto (qs \mapsto rt \rhd qs) &
rt & \mapsto (q \mapsto rt \rhd q) &
rt & \mapsto (qs \mapsto qs \rhd rt)
\end{align*}
  
%\subsection{The tracial pre-Lie-Rinehart algebra of aromatic trees}

\begin{proposition}
  $\lri_{\C}$ is a pre-Lie Rinehart algebra over the commutative algebra $R_{\C}$ spanned by multi-aromas, with anchor map $\rho$ and connection $\nabla$ 
  defined above.
%The bracket is defined by $[\![s,t]\!]:=s\rhd t-t\rhd s$.
\end{proposition}
\begin{proof}
Checking the Leibniz rule for the anchor map and the left pre-Lie relation for $\rhd$ is an easy exercise left to the reader.
\end{proof}

\subsection{The algebra of marked aromatic trees}

\begin{definition} Let $\hat{E}_\C$ be the free $R_\C$-module spanned
  by all pairs $(v,rt)$ where $rt$ is an aromatic tree and $v$ is a vertex
  of $rt$. It identifies naturally as an $R_\C$-submodule of 
  $\End_{R_\C}(L_\C)$ by
  \[ (v,rt) \mapsto (u \mapsto u \rhd_v rt). \]
  In fact by composition $\circ$ it is an $R_\C$-submodule subalgebra.
\end{definition}

Extending the map $\tau$ of \eqref{eq:trmap} we get an $R_\C$-linear map
\[ \tau : \hat{E}_\C \pil R_\C. \]
It maps an element $(v,rt)$ to the multi-aroma we get by joining the root
of $t$ to the vertex $v$. 

\begin{lemma} \label{lem:aroma-Eh}$\hbox{}$
  
  \begin{itemize}
  \item[a.] $\hat{E}_\C$ is an $L_\C$-submodule of $\End_{R_\C}(L_\C)$,
  \item[b.] $\tau$ is an $L_\C$-module map,
  \item[c.] $\tau$ is a trace map, i.e. it vanishes on commutators.
\end{itemize}
  \end{lemma}

\begin{proof}
  a. For an aromatic tree $rt$ and a  marked aromatic tree
  $\alpha=(v,qs)$, one obtains $rt\rhd \alpha$ by grafting the root of $t$ on any vertex of $qs$ and by summing up all possibilities, keeping of course the marked vertex $v$ in each term.

  b. Consider the operations:
  \begin{itemize}
  \item First graft $t$ on each vertex of $qs$ and
    then attach the root of $s$ to $v$. This gives
    $\tau (rt \rhd (v,qs))$.
  \item First attach the root of $s$ on $v$ and then graft
    the root of $t$ to any vertex of the resulting aroma. This
    gives $rt \rhd \tau (v,qs)$.
  \end{itemize}
  We see these operations give the same result, and so
  \[ \tau (rt \rhd (v,qs)) = rt \rhd \tau (v,qs). \]

  c. Look at $\tau ((w,rt) \circ (v,qs))$. We get this by:
  \begin{itemize}
  \item[1.] Composition $\circ$: Graft the root of $s$ on $w$.
  \item[2.] Map $\tau$: Graft the root of $t$ on $v$.
  \end{itemize}
  But the order of these two operations can be switched without
  changing the result. Hence $\tau$ is a trace map.
  \end{proof}

\subsection{The algebra of elementary endomorphisms}

\ignore{
\begin{proposition}
The complexes \eqref{eq:tree-aromad} and \eqref{eq:tree-aromatau} are exact.
\end{proposition}
}

Let $\varo_\C$ be the vector space freely generated by pairs $(v,a)$ where
$a$ is an aroma in $\aro_\C$ and $v$ a vertex of $a$.
There is an injection
\[ \varo_\C \pil \Hom_{\kk}(\tre_C, \aro_\C), \quad
  (v,a) \mapsto (s \mapsto s \rhd_v a), \]
where the tree $s$ is grafted on the vertex $v$ in $a$.
The $T_\C$-module structure on $A_\C$ gives by adjunction an injective
linear map
\[ \di :  \aro_\C \pil \Hom_{\kk}(\tre_\C, \aro_C), \quad a \mapsto (t \mapsto t \rhd a).\]
Its image lies in the image of $\varo_\C$. Thus $\aro_\C$ may be considered
a subspace of $\varo_\C$, identifying
$\di a = \sum_{v \in a} (v,a)$. Let $\arot_\C$ be the $\kk$-vector space
generated by expressions $da \cdot t$, where $a$ is an aroma and $t$ a tree.
It identifies as $dA_\C \te_{\kk} T_\C$. 
%%-----------------------------------

We have compositions:
\begin{align} \notag
  \vtre_\C &\circ \vtre_\C \pil \vtre_\C &:&& (v,s) \circ (w,t)=&
(w, t \rhd_w s) 
  \\  \notag
  \arot_\C &\circ \vtre_\C &:&& \, (da \cdot t) \circ (v,s)    =&
\sum_{u \in a} (v,s \rhd_u a) \cdot t \\
 \label{eq:aroma-DV} &&&&&\\
  \notag
  \vtre_\C &\circ \arot_\C \pil \arot_\C   & :&&\, (v,s) \circ (da \cdot t)
\mapsto
                         & da \cdot (t \rhd_v s) \\\notag
  \arot_\C &\circ \arot_\C \pil R_\C \te_{\kk} \arot_\C & : &&\,
                          (da \cdot t) \circ (db \cdot s)
                          \mapsto &(s \rhd a) \cdot  db \cdot t
\end{align}

%Let $E_\C$ be the $R_\C$-submodule
%  of $\hat{E}_\C$ generated by $\vtre_\C, \arot_\C$ and
%  $\vtre_\C \circ \arot_\C$.

\begin{proposition} \label{prop:aroma-EC}
  The $R_\C$-submodule $E_\C$ of $\hat{E}_\C$
  generated by $\arot_\C, \vtre_\C$ and
  their composition $\arot_\C \circ \vtre_\C$, is the algebra
  of elementary endomorphisms $\emo_{R_\C}(L_\C)$.

  It decomposes as a free $R_\C$-module:
  \begin{equation} \label{eq:aroma-Efree}
    E_\C = \big(R_\C \te_{\kk} V_\C\big) \bigoplus \big(R_\C \te_{\kk} D_\C\big)
    \bigoplus \big(R_\C \te_\kk (D_\C \circ V_\C)\big).
  \end{equation}
    
\end{proposition}

\begin{proof}
  By the relations \eqref{eq:aroma-DV} above, we see that $E_\C$ is an
  $R_\C$-module subalgebra of $\hat{E}_\C$.

  \medskip
\noindent {\it Inclusion $\emo_{R_\C}(L_\C) \sus E_\C$:}
  Considering the map $d$:
  \[ d(rt) = dr \cdot t + r \cdot dt. \]
Looking at the right side of this, 
the first term is in $\arot_\C$, and the second term
in $R_\C \te_{\kk} \vtre_\C$.

\medskip
\noindent {\it Inclusion $E_\C \sus \emo_{R_\C}(L_\C)$:}
  By Lemma \ref{lem:EC}, $\vtre_\C \sus \emo_{R_\C}(L_\C)$.
  Since
  \[ d(at) - a dt = da \cdot t, \]
  also $\arot_\C \sus \emo_{R_\C}(L_\C)$.

  \medskip
  \noindent {\em Free decomposition:}
  %The algebra $R_\C$ has a very fine
  %grading given by multisets of aromas
  %(each graded part being one-dimensional).
  %By forgetting the marked points, to every generator of
  %$E_\C$ we may associate an aromatic tree $rt$, and so $E_\C$ is homogeneous
  %for this grading by aromatic trees. Any $R$-linear relation between
  %generators of $E_\C$ can then be reduced to relations homogeneous for this
  %fine grading.
  An element of $R_\C \te_\kk V_\C$ has its marks on trees, and
  so cannot be an $R$-linear
  combination of the two other parts. Any element of $R_\C \te_\kk D_\C$ must
  have a term with a mark on the interior cycle of an aroma and so cannot
  be a sum of terms in $D_\C \circ V_\C$.
\end{proof}

By Lemma \ref{lem:aroma-Eh} and Proposition \ref{prop:aroma-EC} above we have:  
\begin{corollary}\label{lemma-trace}
The pre-Lie-Rinehart algebra $\lri_{\C}$ of aromatic trees is tracial.
\end{corollary}

%%%%
\section{The universal tracial pre-Lie-Rinehart algebra}
%%%%

We show that the pair $(L_\C, R_\C)$ is
%\[ \lri_C \mto{d} \emo_C \mto{\tau} \rri_C \]
a {\it universal tracial} pre-Lie Rinehart algebra.
\begin{remark}
  Originally we aimed to show that the pair $(\lri_\C, \rri_\C)$
  was a universal pre-Lie-Rinehart algebra. However from a given
  map of sets
  \[ \C \pil L \]
  we could not extend this to maps
 \[ \lri_\C \pil  L, \quad  \rri_\C \pil R. \]  
 The problem is that one cannot generate all of $\lri_\C$ or $\rri_\C$
 by starting
from $\C$ and using the operations $\operatorname{Div} = \tau \circ d$
and $\rhd$ applied on the algebra $L_\C$, either
between aromatic trees $s \rhd t$ or on an aroma $s \rhd a$. In particular
one cannot generate all of the multi-aromas $\rri_\C$.

To remedy this we have introduced the subalgebra $\vtre_\C$ of
$\eend_{\kk}(\tre_\C)$
{\it generated} by the image of $\tre_\C \pil \eend_{\kk}(\tre_\C)$
together with its trace map $\tau$.
From this subalgebra one can get all aromas by applying the trace map.
Furthermore $\vtre_\C$ is ``fattened up'' to the subalgebra
$\emo_\C$ of $\eend_{\rri_\C}(\lri_\C)$ over $\rri_\C$.
To get the universality property we have therefore introduced the class of
tracial pre-Lie-Rinehart algebras.
%natural additional notion for pre-Lie-Rinehart algebras.
\end{remark}

The map $\gamma$ of Theorem \ref{thm:TEA} extends to $\hat{\gamma}: R_\C \to R$
given by
\begin{equation*}
\hat{\gamma}(a_1\cdots a_p):=\gamma(a_1)\cdots\gamma(a_p).
\end{equation*}
The map $\alpha$ of Theorem \ref{thm:TEA} extends to
$\hat{\alpha}:\lri_{\C}\to L$ given by
\begin{equation*}
  \hat{\alpha}(a_1\cdots a_it):= \hat{\gamma}(a_1 \cdots a_p) \alpha(t) = 
  \gamma(a_1)\cdots\gamma(a_i)\alpha(t)
\end{equation*}
for any $a_1,\ldots, a_i\in \aro_{\C}$ and $t\in \tre_{\C}$,.

\begin{theorem}[{\bf Universality property}]  \label{thm:UniE}
  Let $(L,R)$ be a tracial pre-Lie-Rinehart algebra,
  and $\C \to L$ a map of sets.
  \begin{itemize}
    \item[a.] This extends to a unique homomorphism
  of tracial pre-Lie-Rinehart algebras:
  \[ (\hat{\alpha},\hat{\gamma}) : (L_\C, R_\C) \pil (L,R). \]
\item[b.] The map $\beta$ of Theorem \ref{thm:TEA} extends
  to a homomorphism $\hat{\beta}$ of associative algebras
  giving a commutative diagram
 \begin{equation}\label{eq:univ-PEA} \begin{CD}
\lri_{\C} @>{d}>> \emo_{\C} @>{\tau}>> \rri_{\C}\\
@V{\hat{\alpha}}VV  @V{\hat{\beta}}VV @V{\hat{\gamma}}VV \\
L   @>{d}>> \emo_R(L) @>{\tau}>> R.
   \end{CD}
 \end{equation}
 \item[c.]
 It fulfills the following for $u \in L_\C$ and
 $\phi \in E_\C$:
 \begin{itemize}
\item[i.] $\hat{\beta}(\phi)(\hat{\alpha}(u)) = \hat{\alpha}(\phi(u))$, 
\item[ii.] $\hat{\beta}(u \rhd \phi) = \hat{\alpha}(u) \rhd \hat{\beta}(\phi).$
\end{itemize}
\end{itemize}
\end{theorem}

\begin{proof} {\bf Part a.} We show:
  \begin{itemize}
  \item[ai.]  $\hat{\gamma}$ is a $k$-algebra homomorphism,
    \item[aii.] For an aromatic tree $rt$ and a multiaroma $q$:
      $\hat{\gamma}(rt \rhd q) = \hat{\alpha}(rt) \rhd \hat{\gamma}(q)$.
      Note: It is to establish this property that we require the trace
      map $\tau$ to be an $L$-module homomorphism.
\item[aiii.] $\hat{\alpha}$ is a homomorphism of pre-Lie algebras,
\item[aiv.] For a multiaroma $q$ and an aromatic tree $rt$:
$\hat{\alpha}(q \cdot rt) = \hat{\gamma}(q) \cdot \hat{\alpha}(rt)$.
  \item[av.] Uniqueness of $\hat{\alpha}$ and $\hat{\gamma}$.
\end{itemize}
Property ai is by construction since $R_\C$ is a free commutative algebra.
Property aiv is by definition of $\hat{\alpha}$.

Since the action of $\rhd$ of $L_\C$ on $R_\C$ is a derivation (the
anchor map), it is enough for Property aii to show for an aroma $a$ and
tree $t$ that:
\[ \gamma (t \rhd a) = \alpha(t) \rhd \gamma(a) \]
and this is done in Theorem \ref{thm:TEA}.

For Property aiii we have for multiaromas $r,q$ and trees $t,s$ that
  \begin{eqnarray*} \hat{\alpha}(rt \rhd qs) & = & 
\hat{\alpha} \big(r (t\rhd q) s + rq (t \rhd s)\big), \\
& = & \hat{\gamma}\big(r (t\rhd q)\big) \alpha(s) +
      \hat{\gamma}(rq) \alpha(t \rhd s)\\
& = & \hat{\gamma}(r)({\alpha}(t) \rhd \hat{\gamma}(q)) \alpha(s) +
\hat{\gamma}(r)\hat{\gamma}(q) (\alpha(t) \rhd \alpha(s)).        
  \end{eqnarray*}
This again equals:
\[ \hat{\alpha}(rt) \rhd \hat{\alpha}(qs) =
\hat{\gamma}(r) \alpha(t) \rhd  \hat{\gamma}(q) \alpha(s). \]
%  is a consequence of the defintion of $\hat{\alpha}$, Property b,
%  and that $\alpha$ is a homomophism or pre-Lie algebras.

The uniqueness, Property av, of $\hat{\alpha}$ is by $L_\C$ being the
free pre-Lie algebra. As for $\hat{\gamma}$ it is determined by its
restriction $A_\C \pil R$. By the requirement of Definition \ref{def:LR-TF}
and the uniqueness
of $\gamma$ for making a commutative diagram in Theorem \ref{thm:TEA}
we see that the $\hat{\gamma}$ restricted to $A_\C$ must equal $\gamma$.

  \medskip
  \noindent{\bf Part b.}
  %We first b1: Define $\hat{\beta}$. Then we show:
  %\begin{itemize}
  %\item[b2.] $\hat{\beta}$ is an algebra homomorphism,
  %\item[c1.] $\hat{\beta}(\phi)(\hat{\alpha}(u)) = \hat{\alpha}(\phi(u)),$
  %\item[c2.] $ \, \hat{\beta}(t \rhd \phi)
  %    = \hat{\alpha}(t) \rhd \hat{\beta}(\phi).$
  %\end{itemize}
  \noindent Definition of  $\hbeta$: $E_\C$ decomposes as a free
  $R_\C$-module \eqref{eq:aroma-Efree}. We let
  $\hbeta(r \phi) = \hat{\gamma}(r) \hbeta{\phi}$ when $\phi$ is a basis
  element for these free modules.
  On $V_\C$ we let $\hbeta$ be given by $\beta$. On $D_\C$ we define
  \[ \hbeta(da \cdot t) = d\alpha(t) \cdot \alpha(t). \]
  Lastly consider the map
  \[ D_\C \te_{\kk} V_\C \pil D_\C \circ V_\C, \]
where by the latter we mean the vector space spanned by all compositions.
  This map is a bijection. To see this, consider \eqref{eq:aroma-DV}.
  Note that an element $\omega$ in $D_\C \circ V_\C$ has no marked point on the
  interior cycle of the aroma. Let then $v$ be a marked point in
  a term of the element $\omega$
  which has minimal distance from $v$ to the interior cycle. Following the path
  from $v$ to the interior cycle, the vertex attached to the interior cycle
  (but not {\it on} the cycle), must
  be the root of a tree $s$ with $v \in s$,
  which is grafted onto an aroma $a$. Thus we have
  reconstructed $a$, $(v,s)$ and $t$ and can subtract the image
  of a multiple of $(da \cdot t) \circ (v,s)$ from $\omega$. In this way
  we may continue and get $\omega$ as the image of a unique element
  in $D_\C \te_{\kk} V_\C$.

  We may then define $\hbeta$ on $D_\C\te_{\kk} V_\C$ by
  \[ \hbeta \big( (da \cdot t) \circ (v,s) \big) = (d \gamma(a) \cdot \alpha(t))
    \circ \beta(v,s). \]

  \noindent 
 Now we show the homomorphism property of $\hbeta$ . It respects composition of
  $V_\C$ since $\beta$ does. It respects compositions $D_\C \circ V_\C$
  by the above definition. It respects composition $D_\C \circ D_\C$ by
  \begin{eqnarray*}
    \hbeta \big((da \cdot t) \circ (db \cdot s) \big)
    & = & \hbeta \big((s \rhd a) db \cdot t \big) \\
    & = & \gamma(s \rhd a) \hbeta(db \cdot t) \\
    & = & \gamma(s \rhd a) d \gamma(b) \cdot \alpha(t) \\
    & = & (\alpha(s) \rhd \gamma(a)) d \gamma (b) \cdot \alpha(t) \\
    & = & (d \gamma(a) \cdot \alpha(t)) \circ (d\gamma(b) \cdot \alpha(s)) \\
    & = & \hbeta(da \cdot t) \circ \hbeta(db \cdot s).
  \end{eqnarray*}
  Applying $\hbeta$ to the composition $V_\C \circ D_\C$
  \begin{eqnarray*}
    \hbeta\big( (v,s) \circ (da \cdot t)\big)
    & = & \hbeta\big(da \cdot (t \rhd_v s)\big) \\
    & = & d \gamma(a) \cdot \alpha(t \rhd_v s) \\
    \text{(use Part i. of Thm. \ref{thm:TEA})}
    & = & d \gamma(a) \cdot \beta\big( (v,s)) (\alpha(t)\big).
 \end{eqnarray*}
 This map sends $u \in L_\C$ to
 $(u \rhd \gamma(a)) \cdot \beta((v,s))(\alpha(t))$,
 and so does the map
 \[ \hbeta((v,s)) \circ d \gamma(a) \cdot \alpha(t) =
   \hbeta((v,s)) \circ \hbeta(da \cdot t). \]
 So these maps are equal and 
 $\hbeta$ respects composition on $D_\C \circ V_\C$. 
% & = & \hbeta\big((v,s) \circ d \gamma(a) \cdot \alpha(t)\big).
%  \end{eqnarray*}

  \medskip
  \noindent {\bf Part c.} i. For $\phi$ in $V_\C$ this follows easily from
  Part i in Theorem \ref{thm:TEA}. For $\phi$ in $D_\C$ it is an easy
  computation. Since $\hbeta$ respects compositions, we then derive it
  for general $\phi$. 

  \medskip
\noindent ii. When $\phi$ is in $V_\C$ this is by Part ii in Theorem
  \ref{thm:TEA}. When $\phi$ is in $D_\C$ we have the following
  computation using Lemma \ref{prop:LR-dxyr}b :
  \begin{eqnarray*}
    \hbeta \big(t \rhd (da \cdot s)\big)
    & = &
    \hbeta \big( da \cdot (t \rhd s) + d(t \rhd a) \cdot s -
                                               (da\cdot s) \circ dt\big) \\
    & = & d \gamma(a) \cdot \alpha(t \rhd s) + d (\gamma (t \rhd a)) \cdot
        \alpha(s) - d\gamma(a) \cdot \alpha(s) \circ \beta d(t) \\
    & = & d \gamma(a) \cdot (\alpha(t)\rhd \alpha(s)) + d (\alpha (t) \rhd
        \gamma(a)) \cdot
        \alpha(s) - (d\gamma(a) \cdot \alpha(s)) \circ d \alpha(t) \\
    \text{(use Lemma \ref{prop:LR-dxyr}b)}
    & = & \alpha(t) \rhd (d \gamma(a) \cdot \alpha(s)) \\
    & = & \alpha(t) \rhd \hbeta(da \cdot s).
  \end{eqnarray*}
  Now ii follows by the easily checked fact that it holds for
  compositions if it holds for each factor.
\end{proof}

\section{Remarks on equivariance} 
We finally return to some remarks on Theorem~\ref{th:1} in the light of the universal diagram~(\ref{eq:univ-PEA}). 
Consider the canonical example $(L,\triangleright)$ of vector fields on $\RR^d$, where $R = C^\infty(\RR^d)$. 
Let $\C = \{\ab\}$ and choose a mapping $\ab\mapsto f\in L$ inducing the universal arrows $\hat{\alpha}, \hat{\beta},\hat{\gamma}$ in~(\ref{eq:univ-PEA}).
Any affine diffeomorphism $\xi(x) = Ax+b$ on $\RR^d$ induces isomorphisms on $L$, $\eend_R(L,L)$ and $R$ by pullback of tensors:
\begin{align*}
\xi\cdot f &:= A^{-1} f\circ \xi\\
(\xi\cdot G)(f) &:= \xi \cdot(G(\xi\cdot f))\\
\xi \cdot r &:= r\circ \xi
\end{align*}
for $f\in L$, $G\in \eend_R(L,L)$ and $r\in R$. 

Given three finite series $B_L\in \lri_{\C},  B_E\in \emo_{\C},  B_R\in \rri_{\C}$ we obtain three mappings 
\begin{align*} 
\Phi_L(f) &:= \hat{\alpha}(B_L) \colon L \rightarrow L\\ 
  \Phi_E(f) &:= \hat{\beta}(B_E) \colon L \rightarrow \eend_R(L,L) \\
  \Phi_R(f) &:=  \hat{\beta}(B_R) \colon L \rightarrow R. 
\end{align*}
It is straightforward to check that these are all equivariant with respect to the action of affine diffeomorphisms:
$\Phi_L(\xi\cdot f) = \xi\cdot \big(\Phi_L(f)\big)$, $\Phi_E( \xi(f)) = \xi\cdot  \big(\Phi_E(f)\big)$ and $\Phi_R( \xi(f)) = \xi\cdot  \big(\Phi_R(f)\big)$. 
Theorem~\ref{th:1} states that any smooth local  affine equivariant mapping $\Phi\colon L\rightarrow L$ has an aromatic $B$-series $\overline{B_L}\in \overline{ \lri_{\C}}$, where the overline denotes the graded completion, i.e.\ the space of formal infinite series. The proof technique~\cite{munthe2016aromatic}, seems to work also for smooth local mappings between different tensor bundles. Hence, we claim:

\begin{claim}A smooth, local mapping $\Phi_E\colon L\rightarrow \eend_R(L,L)$ has an aromatic series $\overline{B_E}\in \overline{ \emo_{\C}}$ if and only it is affinely equivariant.  A smooth, local mapping $\Phi_R\colon L\rightarrow R$ has an aromatic series $\overline{B_R}\in \overline{ \rri_{\C}}$  if and only it is affinely equivariant. Subject to convergence, the mappings are represented by their aromatic $B$-series.
\end{claim}

\section*{Acknowledgements} The second author thanks Universitetet i Bergen for the warm welcome and stimulating atmosphere during his two visits in September 2018 and May 2019, and Trond Mohn Foundation for support. He also thanks Camille Laurent-Gengoux for an illuminating e-mail discussion on Lie algebroids.
%%%%%%%%

%%%%%
%%%%%%%%%%%%%%%%%%%%%%
\bibliographystyle{plain}
\bibliography{aroma_cite}

%\begin{thebibliography}{abcd}
%
%\bibitem{C1857} A. Cayley,
%	\textsl{On the theory of the analytical forms called trees},
%	Phil. Mag. \textbf{13} No85, 172--176 (1857).
%
%\bibitem{CL2001} F. Chapoton, M. Livernet,
%\textsl{Pre-Lie algebras and the rooted trees operad},
%International Mathematics Research Notices \textbf{2001}, Issue 8, 395--408 (2001).	
%
%\bibitem{DL2002} A. Dzhumadil'daev, C. L\"ofwall,
%\textsl{Trees, free right-symmetric algebras, free Novikov algebras and identities},
%Homology, Homotopy and Appl.\textbf{4} No. 2, 165--190 (2002).
%
%\bibitem{R1963} G. S. Rinehart, 
%\textsl{Differential forms on general commutative algebras}, 
%Trans. Amer. Math. Soc. \textbf{108}, 195--222(1963).
%\end{thebibliography}
%

\end{document}